\documentclass[12pt]{amsart}
\usepackage{amssymb,amsthm,amsmath,amsfonts}
\usepackage{hyperref}
\usepackage{color,verbatim}
\usepackage{multirow}
\def\R{{\mathbb R}}

\def\<{\langle}
\def\>{\rangle}

\def\X{{\mathrm X}}

\def\XB{{\mathrm{\mathbf{X}}}}

\def\Y{{\mathrm Y}}

\def\XXX{\widetilde{{\mathrm{\mathbf{X}}}}}
\def\x{x}

\def\H{\mathcal{H}}

\newtheorem{th-def}{Theorem-Definition}[section]
\newtheorem{theo}{Theorem}[section]

\newtheorem{prop}[theo]{Proposition}

\title{A test for normality and independence based on characteristic function }
\author[Wiktor Ejsmont]{Wiktor Ejsmont}

\address[Wiktor Ejsmont]
{
Mathematical Institute, University of Wroc\l aw \\
pl. Grunwaldzki 2/4, 50-384 Wroc\l aw, Poland
}
\email{wiktor.ejsmont@math.uni.wroc.pl}

\author[Bojana Milo\v{s}evi\'{c}]{Bojana Milo\v{s}evi\'{c}}
\address[Bojana Milo\v{s}evi\'{c}]
{
Faculty of Mathematics, University of Belgrade, Studentski trg 16, Belgrade
}
\email{bojana@matf.bg.ac.rs}

\author[Marko Obradovi\'{c}]{Marko Obradovi\'{c}}
\address[Marko Obradovi\'{c}]
{
Faculty of Mathematics, University of Belgrade, Studentski trg 16, Belgrade
}
\email{marcone@matf.bg.ac.rs}
\subjclass[2010]{Primary: 62H15 Secondary: 62E10.}
\keywords{characterization, multivariate normal distribution, goodness-of-fit test, empirical characteristic function}
\begin{document}

\begin{abstract}
In this article we prove a generalization of the Ejsmont characterization \cite{Ejs} of the multivariate normal distribution. Based on it, we propose a new test for independence and normality. The test uses an integral of the squared modulus of the
difference between the product of  empirical  characteristic functions and some constant. Special attention is given to  the case of testing univariate normality in which we derive the test statistic explicitly in terms of Bessel function, and the case of testing bivariate normality and independence. The tests show quality performance in comparison to some popular powerful competitors.
\end{abstract}
\maketitle

\section{Introduction}
One of classical and important problems in statistics is testing independence between two
of more components of a random vector. The traditional approach is based on Pearson's
correlation coefficient, but its lack of robustness to outliers and departures from normality
eventually led researchers to consider alternative nonparametric procedures. To overcome
this problem, some rank tests of independence are proposed, such as Savage, Spearman and van der Waerden, that in particular rely on linear rank statistics. The present paper uses
another way to test the independence and normality that is based on a distance between the empirical function and a constant  $e^{-\frac{1}{2}}$. 

Many statistical studies deal with the relationship between two random vectors, say   $(\X_1,\dots, \X_m)$ and $(\Y_{1},\dots,\Y_n)$,  and in particular, with the question whether random variables $\X_i$ and $\Y_j$  are independent and have the same normal distribution (see \cite{Kallenberg}). \\ Assuming multivariate normality of $$(\X_1,\dots, \X_m,\Y_{1},\dots,\Y_n)$$ the problem reduces to testing the null hypothesis that the correlation coefficients are equal to $0$. 
Indeed, it is well known from the general theory of probability that
if a random vector has a multivariate normal distribution (joint normality), then any two or more of its components that are uncorrelated, are independent. This implies that any two or more of its components that are pairwise independent, are independent.  

A theoretical framework to study this aspect in a general sense was given by Ejsmont \cite{Ejs}. Ejsmont proved that the characterizations
of a normal law are given by a certain invariance of the noncentral chi-square distribution.
Namely in  \cite{Ejs}  it has been shown that if the random vectors 
$(\X_1,\dots, \X_m)$ and $(\Y_{1},\dots,\Y_n)$ are independent with all moments, and the distribution
of $\sum_{i=1}^m\X_ia_i+A+\sum_{j=1}^n\Y_jb_j+B$ depends only on $\sum_{i=1}^ma_i^2+\sum_{j=1}^nb_j^2$, then  $\X_1,\dots, \X_m,\Y_{1},\dots,\Y_n$  are independent and have the same normal distribution. 
In the above result we especially remove the requirement that corresponding random variables are independent and have the same distribution. 

 The paper is organized as follows. In Section 2 we state and prove the main results of  \cite{Ejs} under the weakened assumption. 
Next, in Section 3 using this result we propose a new test for normality. Finally, in section 4 we obtain the explicit representation for two random vectors in the language of Bessel functions of the first kind, and in this case we simulate critical values.

\section{The theoretical base for the construction of a test}
\textit{Notation.}  The scalar product of vectors $t,s\in\R^p$ is denoted by $\langle t, s \rangle$ and the Euclidean norm of $t$  is $\|t\|=\sqrt{\langle t, t \rangle}$. Throughout this paper $\X:=(\X_1,\dots, \X_m) \in \R^m$ and $\Y:=(\Y_{1},\dots,\Y_n)\in \R^n$ are random vectors, where $m$
and $n$ are positive integers. The characteristic functions of $\X$ and $\Y$ are denoted by
$\varphi_\X(\cdot)=Ee^{i\langle \cdot,\X \rangle }$ and $\varphi_\Y(\cdot) =Ee^{i\langle \cdot,\Y \rangle }$, respectively. For
complex-valued functions $f(\cdot)$, the complex conjugate of $f$ is denoted by $\overline{f}$ and
$|f|^2 = f \overline{f}$. 
 In order to simplify notation, we will denote  $[n]=\{1,\dots,n\}$.
 We  denote  $(a,b) \in \R^{m+n}$ the  concatenation  of  the vectors  $a\in \R^m$ and  $b\in \R^n.$

Our construction of a new test of normality is based on the  following result.   
This is a generalization of the main result of \cite{Ejs}, under omitted moment assumptions (in \cite{Ejs} we assume that random  variables have all moments; our proof is also  different). 
\begin{theo}
Let $(\X_1,\dots, \X_m,A) \textrm{ and } (\Y_{1},\dots,\Y_n,B)$ be independent random vectors, where  $\X_i$ and $\Y_j $  are 
nondegenerate for $i\in[m],j\in[n]$, and let statistic $$\langle a,\X\rangle +\langle b,\Y\rangle+A+B=\sum_{i=1}^ma_i\X_i+\sum_{j=1}^n b_j\Y_j+A+B,$$
have a distribution which depends only on $\|a\|^2+\|b\|^2$, where $a\in \mathbb{R}^m$ and $b\in \mathbb{R}^n$.  
Then random variables $\X_1,\dots,\X_m,\Y_1,\dots,\Y_n$ are independent and have the same normal distribution with zero means. 
\label{tw:Main1}
\end{theo}

\begin{proof} Our proof is based on the analysis of the characteristic function and so we denote by $\varphi_{W}(\cdot)$ the characteristic function
of $W$. 
We write $(a,b)=r(\tilde{a},\tilde{b})$ where $(\tilde{a},\tilde{b})$ belongs to
the unit sphere of $\R^{n+m}$  i.e., $r=\sqrt{\|a\|^2+\|b\|^2}$. 
Thus for 
$r>0$ and $t\in \R$, we have
\begin{align}\label{rownosccharekterys}
\varphi_{\frac{\langle a,\X\rangle +\langle b,\Y\rangle+A+B}{r}}(t)=\varphi_{\langle \tilde{a},\X\rangle +\langle \tilde{b},\Y\rangle+\frac{A+B}{r}}(t). 
\end{align} 
By the hypothesis the left hand side of \eqref{rownosccharekterys} does not depend on $(\tilde{a},\tilde{b})$ (= depend on $r$), and thus the limit on the
right hand side
then
$$\varphi_{\langle \tilde{a},\X\rangle +\langle \tilde{b},\Y\rangle+\frac{A+B}{r}}(t)\xrightarrow{r\to + \infty}\varphi_{\langle \tilde{a},\X\rangle +\langle \tilde{b},\Y\rangle}(t),
$$
does not depend on $(\tilde{a},\tilde{b})$.
In particular, we have that the distribution of a statistic
\begin{align}
&\langle {a},\X\rangle +\langle {b},\Y\rangle =(\langle \tilde{a},\X\rangle +\langle \tilde{b},\Y\rangle)\sqrt{\|a\|^2+\|b\|^2} \nonumber
\intertext{dependence on $\|{a}\|^2+\|{b}\|^2$ only. Let} &h(\|{a}\|^2+\|{b}\|^2)=Ee^{i(\langle a,\X\rangle +\langle b,\Y\rangle)}.\nonumber
 \intertext{Because of the independence of $\X $ and $\Y$,  we may write }
& h(\|{a}\|^2+\|{b}\|^2)=Ee^{i\langle a,\X \rangle } Ee^{i\langle b,\Y\rangle}=\varphi_\X(a)\varphi_\Y(b). \label{eq:niezalensoc} 
\intertext{Evaluating \eqref{eq:niezalensoc} first when $a= \mathbf{ 0}\in\R^m$ and then when $b=\mathbf{ 0}\in\R^n$, we get}\nonumber
&h(\|{b}\|^2)=\varphi_\Y(b)\text{ and }h(\|{a}\|^2)=\varphi_\X(a), 
\intertext{respectively. Substituting this into \eqref{eq:niezalensoc}, we obtain}
& h(\|{a}\|^2+\|{b}\|^2)=h(\|{a}\|^2)h(\|{b}\|^2).\nonumber
\intertext{Note that $h(\cdot)$ is continuous, hence by multiplicative Cauchy functional equation we get }
&h\left(\|{a}\|^2+\|{b}\|^2\right)=e^{c(\|{a}\|^2+\|{b}\|^2)}.\nonumber
\intertext{Substituting  $a=(a_1,0,0,\dots,0)$ and $b=\mathbf{ 0}$ in this equation, we see that it can be read as }
&Ee^{i\X_1a_1}=e^{ca_1^2}, \text{ i.e. $\X_1$ have a normal distribution, with zero mean.}\nonumber
\intertext{Dragging this line of reasoning to other random variables,  
 we see that $\X_i$ and $\Y_j$  have the same normal distribution, with zero mean.
 The independence of random variables $\X_1,\dots, \X_m$ follows from the observation that, for all  $a=(a_1,\dots,a_m)\in\R^m$}
&
\varphi_\X(a)=e^{{c\sum_{j=1}^m{a_j^2}}}=\varphi_{\X_1}(a_1) \dots \varphi_{\X_m}(a_m).\nonumber
\end{align} 
\end{proof}

The construction of a new test is based directly on the Proposition below, that follows, in a sense, from Theorem \ref{tw:Main1}, namely if $A=B=0$, then Theorem \ref{tw:Main1} can be rewritten as follows.
\begin{prop} 
 Let 
$(\X_1,\dots, \X_m) \textrm{ and } (\Y_{1},\dots,\Y_n)$ be independent random vectors, where  $\X_i$  and  $\Y_j$ are
nondegenerate, $E(\X_i^2)=1$,  $E(\Y_j^2)=1 $ for $i\in [m]$, $j\in [n]$.  
Then the following statements are equivalent:
\begin{enumerate}
\item[i)] statistic $\langle a,\X\rangle +\langle b,\Y\rangle$ has a distribution which does not depend on $$(a_1,\dots, a_m,b_1,\dots,b_n),$$ whenever $\|a\|^2+\|b\|^2=1$;
\item[ii)]    random variables $\X_1,\dots,\X_m,\Y_1,\dots,\Y_n$  are independent and have the same normal distribution $N(0,1)$.
\end{enumerate}
\label{rem1}
\end{prop}

\begin{proof}
$(i)\Rightarrow (ii)$. We see that the distribution of
$$\langle a,\X\rangle +\langle b,\Y\rangle=\sqrt{\|a\|^2+\|b\|^2}\frac{\langle a,\X\rangle +\langle b,\Y\rangle}{\sqrt{\|a\|^2+\|b\|^2}}$$
depends only on $\|a\|^2+\|b\|^2$, which by Theorem  \ref{tw:Main1} implies that 
$\X_i$ and $\Y_j$  are independent and have the same normal distribution $N(0,1)$ (because we assume that $E(\X_i^2)=1$,  $E(\Y_j^2)=1$). 
\newline 
$(ii)\Rightarrow (i)$. We compute the characteristic function 
 $$Ee^{i\langle a,\X\rangle+i\langle b,\Y\rangle}=e^{-(\|a\|^2+\|b\|^2)/2},$$
 from which we see that condition $(i)$ is satisfied. 
 \newline 
\end{proof}

\section{The test statistic } 
In this section we propose a  new class of test  statistics for testing the null hypothesis that the sample comes from a multivariate normal distribution with independent components. In the univarate case it reduces to the null normality hypothesis.

Our methodology applied in this construction is based on distances between empirical and theoretical quantities.  
There are many types of distances in theory of hypothesis testing that can be defined between statistical objects. One of the best known and mostly applied is the $L_2$ distance. If $F$ is the cumulative distribution function (cdf) of a random variable and $F_n$ is the empirical function (edf), then their $L_2$ distance can be expressed as $\int_{-\infty}^{\infty}(F_n(x)-F(x))^2dx$,  introduced by Cram\'er \cite{Cramer}. 
Later modifications of this distance lead to Cram\'er-von Mises test and to Kolmogorov-Smirnov test \cite{Smirnov}. 
There is, however, another important distance, if the sample comes from a $d$-dimensional space, where $d\geq 1$.
 If we want to test multivariate normality then we can use the distance between empirical and theoretical characteristic function; see \cite{Bar,Epps}.
More recently, the characterization of a test for multivariate independence was given in \cite{Szekly,Szekly2}. Suppose that $\X\in\R^m,\Y\in\R^n$ are real-valued random vectors with characteristic functions  ${\varphi_X}$ and ${\varphi_\Y}$,
respectively. Then, for measuring independence, we can use the following distance 
$\int_{\R^{m+n}}|{\varphi_{\X,\Y}}(t,s)-{\varphi_\X}(t){\varphi_\Y}(s)|^2w(t,s)dtds$, where $w(t, s)$ is an arbitrary positive weight function for which the integral above exists. We put forward a test that is also  based  on the distance between the characteristic function and some constant, and it was inspired by the articles \cite{Bar,Szekly,Szekly2,Epps}.
 Our approach is based on the following reasoning.
\newline
\newline
Condition $(i)$ from Proposition \ref{rem1} simply tells us that we get statement $(ii)$ if the distribution of the statistic 
$\langle a,\X\rangle+\langle b,\Y\rangle$ is constant on the  $(n+m)$-sphere with radius $1$.
This requirement
can be rewritten using the characteristic function, namely,  we get statement $(ii)$ if and only if the  function 
\begin{align}
 \nonumber & Ee^{i\langle a,\X\rangle+i\langle b,\Y\rangle} =\varphi_\X(a)\varphi_\Y(b) 
\intertext{is constant on the unit sphere $\|a\|^2+\|b\|^2=1$,  where $a\in \mathbb{R}^m$ and $b\in \mathbb{R}^n$.  
 From the proof of Proposition \ref{rem1} we also know that this constant function must equal $e^{-\frac{1}{2}}$,  namely}
 &\varphi_\X(a)\varphi_\Y(b)-e^{-\frac{1}{2}}=0 \nonumber
 \intertext{ for all $\|a\|^2+\|b\|^2=1$ or equivalently,  }
 &\int_{S_{n+m}}|\varphi_\X(a)\varphi_\Y(b)-e^{-\frac{1}{2}}|^2dS_{n+m}=0,
\label{pstaccalkowa}
 \intertext{where $\int_{S_{n+m}} \cdot dS_{n+m}$ is the surface integral over   $S_{n+m}=\{t\in\R^{n+m}\mid \|t\|=1\}$.
Finiteness of the integral above follows directly from $|\varphi_\X(a)\varphi_\Y(b)|\leq 1$ and $e^{-\frac{1}{2}}<1$, namely   we see that }
&  \int_{S_{n+m}}|\varphi_\X(a)\varphi_\Y(b)-e^{-\frac{1}{2}}|^2dS_{n+m}
\leq (1-e^{-\frac{1}{2}})^2|S_{n+m}|. \nonumber
\end{align}

Let us assume that we have a simple random sample $\XB=(\boldsymbol{X}_1,\ldots, \boldsymbol{X}_N)$ from a multivariate distribution with $m$ components, i.e. the data have the following structure: 
    \begin{equation*}
   \XB= 
      \begin{bmatrix}
        \x_{1,1} & \x_{1,2} & \dots  &  \x_{1,m}  \\ 
        x_{2,1} &  \x_{2,2} & \dots & \x_{2,m} 
        \\
        \multicolumn{4}{c}{\dotfill}\\
        x_{N,1} &  \x_{N,2} & \dots &   \x_{N,m} 
      \end{bmatrix}.
    \end{equation*}

We want to test the null hypothesis
\newline 
\begin{center}
\begin{tabular}{ c c c }
 $\H^{(m)}_0$: & & $\H^{(m)}_1$:\\
 all columns of  $\XB$  are independent,    & \qquad vs.\qquad \quad&  $\H^{(m)}_0$  is not true. \\ 
 and have a normal distribution &  &  \\  
\end{tabular}
\end{center}

  Let $\XXX$ denote the matrix obtained from $\XB$ by columnwise standardization.
Let  
$\varphi_{\XXX}(a)$  be the empirical  characteristic functions of  $\XXX$ defined by 
$$\varphi_{\XXX}(a)=\frac{1}{N}\sum_{k=1}^N e^{i \langle a,\XXX_k\rangle},$$
 where $\XXX_k$ is the $k$th row of the matricex $\XXX$.
 Similarly, the empirical counterpart of characteristic function of random variable $\langle a,X\rangle+\langle b,Y\rangle$ is
 \begin{align*}
     \frac{1}{N^2}\sum_{j,k}e^{i(\langle a,{\boldsymbol{X}}_j\rangle +\langle b,{\boldsymbol{X}}_k\rangle )},
 \end{align*}
   where $a=(a_1,...,a_m)^{\textup{T}}$ and $b=(b_1,...,b_m)^{\textup{T}}$.
 Assuming that $X$ and $Y$ are equally distributed as $\boldsymbol{X}_1$, the natural test statistics based on \eqref{pstaccalkowa} is
 

\begin{align}\label{Mm}M_{m}&=N\int\left|\frac{1}{N^2}\sum_{j,k}e^{i(\langle a,\widetilde{\boldsymbol{X}}_j\rangle +\langle b,\widetilde{\boldsymbol{X}}_k\rangle)}-e^{-\frac{1}{2}}\right|^2dS_{2m}(a,b),
\end{align}
    where $a=(a_1,...,a_m)^{\textup{T}}$ and $b=(b_1,...,b_m)^{\textup{T}}$ such that $\langle a,a\rangle+\langle b,b\rangle=1$.
 
Clearly, we are interested in one-sided test, that is the right-tailed test, because we see from the above construction that we reject null hypothesis for large values of  $M_m$. It is clear that test statistic is location-scale invariant under the null hypothesis, hence we may derive critical values for testing using Monte Carlo approach.

\section{Testing univariate normality} 

Consider now a univariate simple random sample $\boldsymbol{X}=(X_1,\ldots,X_n)$. In this case test statistic $M_1$ can be expressed in a simpler form. 

\begin{prop} Let $\boldsymbol{\tilde{X}}=(\tilde{X}_1,\ldots,\tilde{X}_n)$ be the standardized sample. 
The  statistic  $M_1$ has the form 

     \begin{align} \label{M1} 
M_{1}&=2\pi N\Bigg[\frac{1}{N^4}\sum_{n,j,k,l=1}^N J(d(\tilde{X}_{n}-\tilde{X}_k, \tilde{X}_j-\tilde{X}_l))-e^{-\frac{1}{2}}\frac{2}{N^2}\sum_{n,j=1}^N J(d(\tilde{X}_{n},\tilde{X}_j))+e^{-1}\Bigg]. 
    \end{align}
    where $J$ is the Bessel function (of order zero) of the first kind, namely $J(z)=\sum_{k=0}^\infty(-1)^k\frac{(z^2/4)^k}{(k!)^2}$ and $d$ is the distance from origin to point  $(x,y)$ i.e., $d(x,y)=\sqrt{x^2+y^2}$.
\end{prop}
\begin{proof}Let us calculate the integral
in the right hand side of \eqref{Mm}. 
 \begin{align*} 
M_{1}&=2\pi N\int\left|\frac{1}{N^2}\sum_{i,j}e^{i(a_1\tilde{X}_{i}+a_2\tilde{X}_{j})}-e^{-\frac{1}{2}}\right|^2dS_2(a,b) \\
&=2\pi N\int U^2(a,b) dS_2(a,b). 
    \end{align*}

Since 
\begin{align*}
    U^2(a,b)&=\Big|\frac{1}{N^2}\sum_{i,j}e^{i(a\tilde{X}_i+b\tilde{X}_j)}-e^{-\frac{1}{2}}\Big|^2\\&=\left(\frac{1}{N^2}\sum_{i,j}\cos(a\tilde{X}_i+b\tilde{X}_j))-e^{-\frac{1}{2}}\right)^2+\left(\frac{1}{N^2}\sum_{i,j}\sin(a\tilde{X}_i+b\tilde{X}_j))\right)^2\\
    &=\frac{1}{N^4}\sum_{i,j,k,l}\Bigg(\Big(\cos(a\tilde{X}_i+b\tilde{X}_j)-e^{-\frac{1}{2}}\Big)\Big(\cos(a\tilde{X}_k+b\tilde{X}_l)-e^{-\frac{1}{2}}\Big)\\&+\sin(a\tilde{X}_i+b\tilde{X}_j)\sin(a\tilde{X}_k+b\tilde{X}_l)\Bigg),
    \end{align*}
    switching to polar coordinates we obtain
    \begin{align*}
    U^2(a,b) &=\frac{1}{N^4}\sum_{i,j,k,l}\Big((\cos(a\tilde{X}_i+b\tilde{X}_j))((\cos(a\tilde{X}_k+b\tilde{X}_l)))+\sin(a\tilde{X}_i+b\tilde{X}_j)\sin(a\tilde{X}_k+b\tilde{X}_l)\Big)\\
    &-\frac{2e^{-\frac{1}{2}}}{N^2}\sum_{i,j}\cos(a\tilde{X}_i+b\tilde{X}_j)+e^{-1}
    \\&=\frac{1}{N^4}\sum_{i,j,k,l}(\cos(a(\tilde{X}_i-\tilde{X}_k)+b(\tilde{X}_j-\tilde{X}_l))-\frac{2e^{-\frac{1}{2}}}{N^2}\sum_{i,j}\cos(a\tilde{X}_i+b\tilde{X}_j)+e^{-1}\\&=\frac{1}{N^4}\sum_{i,j,k,l}(\cos(\cos\alpha(\tilde{X}_i-\tilde{X}_k)+\sin\alpha(\tilde{X}_j-\tilde{X}_l))-\frac{2e^{-\frac{1}{2}}}{N^2}\sum_{i,j}\cos(\cos\alpha \tilde{X}_i+b\sin\alpha \tilde{X}_j)+e^{-1},
\end{align*}
Since we need integration over $S_{2}$, we have
to focus on computing
the following integral
\begin{align*}
&\int_{0}^{2\pi}\cos(x\cos \alpha + y \sin \alpha )d\alpha, \qquad\text{ for } x,y\in\R. \end{align*}

By trigonometric
identities the linear combination, or harmonic addition, of sine and cosine waves is equivalent to a single cosine wave with a phase shift and scaled amplitude, namely

\begin{align*}
x\cos \alpha + y \sin \alpha =\sqrt{x^2+y^2}\cos(\alpha-\textup{atan2}(y,x)),
\end{align*}
where $\textup{atan2}(y, x)$ is the generalization of $\arctan(y/x)$ that covers the entire circular range (we don't need a formal definition of $\textup{atan2}$).
Thus for $x,y\in \R$ and $xy\neq 0$ we get  

\begin{align*}
\int_{0}^{2\pi}\cos(x\cos \alpha + y \sin \alpha )d\alpha &=\int_{0}^{2\pi}\cos\big(\sqrt{x^2+y^2}\cos\big(\alpha-\textup{atan2}(y,x)\big)\big)d\alpha\\&=\int_{-\textup{atan2}(y,x)}^{2\pi-\textup{atan2}(y,x)}\cos\Big(\sqrt{x^2+y^2}\cos t \Big)dt\\&=2\pi J(\sqrt{x^2+y^2}),
\end{align*}
where we used the following identity -- see \cite[page 360]{AdrSteg} 
\begin{align}
 2\pi J(z)=\int_{0}^{2\pi}e^{iz\cos \alpha }d\alpha=\int_{0}^{2\pi}\cos(z\cos \alpha )d\alpha=\int_{0}^{2\pi}\cos(z\sin \alpha )d\alpha. \label{eq:stegun}
\end{align}
If either $x=0$ or $y=0$, then the formula above is also true because we can use directly equation \eqref{eq:stegun}.

Therefore $M_1$ has the representation \eqref{M1}.
\end{proof}

In Tables \ref{tab: Univariate1} and \ref{tab: Univariate2} we present power study results for sample sizes $n=20$ and $n=50$. The results are obtained using the Monte Carlo method with $N=5000$  replicates.

    \begin{table}[!h]
		\centering
	\caption{Power comparison for testing univariate normality -- Part I}
    \label{tab: Univariate1}
		\begin{tabular}{|cc|cccccccccc|}\hline
    	
    	Alt. & n & SW & BCMR & BHEP & AD & SF & HJG$_{2.5}$ & HJG$_{5}$ & BE$_{1}^{(1)}$ & BE$_{1}^{(2)}$ & $M_{1}$\\\hline
    	\multirow{3}{*}{N(1,4)} & 20 & 5& 5& 5& 5& 5& 5& 5& 5& 5& 5\\
    	& 50 & 5& 5& 5& 5& 5& 5& 5& 5& 5& 5\\
    	& 100 & 5& 5& 5& 5& 5& 5& 5& 5& 5& 5\\\hline
    		\multirow{3}{*}{MixN(0.3,1,0.25)} & 20 & 28& 28& 27& 30& 25& 11& 13& 24& 20& 10\\
    	& 50 & 60& 60& 62& 68& 57& 16& 26& 56& 48& 20\\
    	& 100 & 89& 89& 90& 94& 88& 28& 49& 87& 78& 38\\
    		\multirow{3}{*}{MixN(0.5,1,4)} & 20 & 40& 43& 42& 46& 48& 34& 33& 36& 33& 19\\
    	& 50 & 78& 80& 80& 86& 83& 49& 49& 63& 46& 36\\
    	& 100 & 97& 98& 98& 99& 98& 69& 68& 91& 66& 57\\\hline
    		\multirow{3}{*}{$t_3$} & 20 & 34& 37& 34& 33& 40& 38& 37& 34& 34& 36\\
    	& 50 & 64& 65& 61& 60& 69& 64& 62& 54& 50& 62\\
    	& 100 & 88& 89& 86& 85& 91& 86& 84& 76& 67& 83\\
    	\multirow{3}{*}{$t_5$} & 20 & 19& 20& 18& 17& 22& 22& 22& 19& 19& 22\\
    	& 50 & 35& 37& 32& 31& 41& 40& 38& 29& 29& 36\\
    	& 100 & 56& 58& 50& 48& 63& 59& 55& 41& 37&56\\
    		\multirow{3}{*}{$t_{10}$} & 20 & 10& 11& 9& 9& 12& 12& 12& 10& 10& 12\\
    	& 50 & 16& 17& 13& 12& 20& 20& 19& 13& 14& 18\\
    	& 100 & 22& 24& 16& 15& 28& 28& 26& 16& 15&24\\\hline
    	\multirow{3}{*}{$\textup{U}(-\sqrt3,\sqrt3)$} & 20 & 21& 17& 13& 17& 8& 0& 0& 3& 2& 1\\
    	& 50 & 75& 70& 55& 58& 47& 0& 0& 7& 2& 0\\
    	& 100 & 100& 99& 95& 95& 97& 0& 0& 32& 3&3 \\\hline
    	\multirow{3}{*}{$\chi^2_5$} & 20 & 43& 44& 42& 38& 42& 33& 36& 44& 44& 42\\
    	& 51 & 88& 88& 84& 80& 85& 65& 76& 87& 86& 86\\
    	& 100 & 100& 100& 99& 99& 100& 91& 98& 99& 99& 100\\
    		\multirow{3}{*}{$\chi^2_{15}$} & 18 & 17& 18& 17& 17& 18& 16& 17& 18& 18& 20\\
    	& 50 & 43& 42& 39& 34& 40& 31& 37& 45& 46& 47\\
    	& 100 & 75& 74& 68& 61& 71& 54& 68& 78& 78&74 \\\hline
    		\multirow{3}{*}{$\textup{B}(1,4)$} & 22 & 60& 60& 53& 53& 54& 30& 35& 52& 49& 41\\
    	& 50 & 98& 98& 94& 95& 97& 57& 76& 94& 92& 90\\
    	& 100 & 100& 100& 100& 100& 100& 89& 99& 100& 100& 100\\
    		\multirow{3}{*}{$\textup{B}(2,5)$} & 20 & 16& 16& 16& 14& 14& 9& 11& 15& 15& 14\\
    	& 50 & 50& 47& 45& 39& 40& 16& 25& 44& 42& 38\\
    	& 100 & 90& 89& 80& 76& 82& 29& 54& 80& 78& 80\\\hline
    		\end{tabular}
\end{table}

    \begin{table}[h!]
		\centering
	\caption{Power comparison for testing univariate normality -- Part II}
    \label{tab: Univariate2}
		\centering
		\begin{tabular}{|cc|cccccccccc|}\hline
    	Alt. & n & SW & BCMR & BHEP & AD & SF & HJG$_{2.5}$ & HJG$_{5}$ & BE$_{1}^{(1)}$ & BE$_{1}^{(2)}$ & $M_{1}$\\\hline
    		\multirow{3}{*}{$\Gamma(1,5)$} & 20 & 83& 83& 77& 77& 80& 57& 63& 78& 76& 72\\
    	& 50 & 100& 100& 100& 100& 100& 91& 97& 100& 100& 99\\
    	& 100 & 100& 100& 100& 100& 100& 100& 100& 100& 100& 100\\
    		\multirow{3}{*}{$\Gamma(5,1)$} & 20 & 24& 24& 23& 20& 24& 20& 22& 25& 25& 25\\
    	& 50 & 60& 59& 53& 49& 58& 42& 50& 63& 62& 62\\
    	& 100 & 90& 90& 85& 81& 88& 69& 83& 91& 91& 93\\\hline
    	\multirow{3}{*}{Gum(1,2)} & 20 & 31& 32& 31& 28& 32& 28& 30& 33& 33& 34\\
    	& 51 & 68& 69& 68& 62& 68& 53& 66& 73& 72& 73\\
    	& 100 & 95& 95& 93& 88& 95& 84& 90& 94& 96& 97\\\hline
    	\multirow{3}{*}{LN(0,1)} & 20 & 93& 93& 91& 90& 91& 78& 83& 91& 90& 87\\
    	& 50 & 100 & 100& 100& 100& 100&  99& 100& 100& 100& 100\\
    	& 100 & 100 & 100& 100& 100& 100& 100& 100&  100& 100& 100\\\hline
    	\end{tabular}
\end{table}

Among the plethora of normality tests we selected to evaluate the performance of our test versus  the most popular normality tests ( the Shapiro-Wilk test (SW), see \cite{shapiro}, the Shapiro-Francia
test (SF), see \cite{shapiroF}, and the Anderson-Darling test (AD), see \cite{ad}). Those tests are implemented in  R package nortest by \cite{nort}.  Additionally, we consider recent powerful  tests based on empirical characteristic function (BHEP), see \cite{bhep},  quantile correlation test based on   the L2 Wasserstein distance, see \cite{bcmr}, the moment generating function ($HJG_{\beta}$) proposed in \cite{hjg} and test based on Stein fixed point characterization  proposed in \cite{Ebner}.

The alternatives we consider are normal mixtures $\textup{MixN}(p,\mu,\sigma^2)=(1-p)N(0,1)+pN(\mu,\sigma^2)$, Student $t_\nu$ distribution, uniform $U(a,b)$ distribution, chi-squared $\chi^2_\nu$, beta $\textup{B}(a,b)$, gamma $\Gamma(a,b)$, Gumbel $\textup{Gum}(\mu,\sigma)$ and lognormal $\textup{LN}(\mu,\sigma)$ where all parameters are standard distribution parameters. This set of alternatives was also used in \cite{Ebner}.

It can be seen form Tables \ref{tab: Univariate1} and \ref{tab: Univariate2} that the powers are reasonably high in comparison to other tests for all alternatives except for the uniform and normal mixtures. In the case of the Gumbel distribution our test outperforms the competitors, and for Gamma and Chi-squared it is one of the best.

\section{Testing bivariate normality and independence}

Consider now a bivariate simple random sample $\XB=(\boldsymbol{X}_1,\ldots, \boldsymbol{X}_N)$, where $\boldsymbol{X}_j=(X_{j1},X_{j2})$, $j=1,\ldots,N$.
Let $\widetilde{\bf{X}}=(\widetilde{\boldsymbol{X}}_1,\ldots, \widetilde{\boldsymbol{X}}_N)$ be its standardization.
Here, the test statistic \eqref{Mm} becomes

\begin{align*}M_{2}&=
N\int\left|\frac{1}{N^2}\sum_{j,k}e^{i(a_1\widetilde{X}_{j1}+a_2\widetilde{X}_{j2} +b_1\widetilde{X}_{k1}+b_2\widetilde{X}_{k2})}-e^{-\frac{1}{2}}\right|^2dS_{4}(a,b) \\&=\int\Bigg(\frac{1}{N^4}\sum_{i,j,k,l}(\cos(a_1(\tilde{X}_{i1}-\tilde{X}_{k1})+a_2(\tilde{X}_{i2}-\tilde{X}_{k2})+b_1(\tilde{X}_{j1}-\tilde{X}_{l1})+b_2(\tilde{X}_{j2}-\tilde{X}_{l2}))\\&-\frac{2e^{-\frac{1}{2}}}{N^2}\sum_{i,j}\cos(a_1\tilde{X}_{i1}+a_2\tilde{X}_{i2}+b_1\tilde{X}_{j1}+b_2\tilde{X}_{j2})+e^{-1}\Bigg)dS_{4}(a,b),
\end{align*}
    where $a=(a_1,a_2)^{\textup{T}}$ and $b=(b_1,b_2)^{\textup{T}}$ such that $\langle a,a\rangle+\langle b,b\rangle=1$.

In Table \ref{tab:bBivPow} we present powers of the new test and the test $KS2$, initially proposed in \cite{Koziol} with data driven parameter selection introduced in \cite{Kallenberg}. We have chosen this competitor since it is the only one in the literature known so far, for testing bivariate normality and independence. In \cite{Kallenberg} it is shown that it outperforms Kolmogorov-Smirnov and Hoeffding test in  most cases. The set of alternatives is taken from \cite{Kallenberg} for some choice of distribution parameters, and is given below.  Unless stated otherwise, all distributions are defined for $x_i\in\mathbb{R}$, $i=1,2$. Distributions derived from the bivariate normal inherit its parameter space $\mu_i\in\mathbb{R},\sigma_i>0$,  $i=1,2,$ $\rho\in[-1,1]$. 

\begin{itemize}
    \item a bivariate normal distribution BivNorm($\mu_1,\mu_2,\sigma_1,\sigma_2,\rho$)  with density
    \begin{align*}
        g_1(x_1,x_2;\mu_1,\mu_2,\sigma_1,\sigma_2,\rho)&=\frac1{2\pi\sigma_1\sigma_2\sqrt{1-\rho^2}}\;\;e^{-\frac{1}{2(1-\rho)^2}\big(\frac{(x_1-\mu_1)^2}{\sigma_1^2}+\frac{(x_2-\mu_2)^2}{\sigma_2^2}-\frac{2\rho(x_1-\mu_1)(x_2-\mu_2)}{\sigma_1\sigma_2}\big)},\\& 
    \end{align*}
     
    \item a mixture of bivariate normal distributions NMixA($\rho$) with density
    \begin{align*}
        g_2(x_1,x_2,\rho)=\frac12 g_1(x_1,x_2;0,0,1,1,\rho)+\frac12 g_1(x_1,x_2;1,1,1,1,0.9);
    \end{align*}
    
    \item a mixture of bivariate normal distributions NMixB($\rho$) with density
    \begin{align*}
        g_3(x_1,x_2,\rho)=\frac12 g_1(x_1,x_2;0,0,1,1,\rho)+\frac12 g_1(x_1,x_2;0,0,1,1,-\rho);
    \end{align*}
    \item a bivariate lognormal distributions LogN($\sigma_1,\sigma_2,\rho$) with density
      \begin{align*}
        g_4(x_1,x_2;\sigma_1,\sigma_2,\rho)=\frac{b_1b_2}{(b_1x_1+a_1)(b_2x_2+a_2)}g_1(l_1,l_2;0,0,\sigma_1,\sigma_2,\rho),\;\;x_i>-\frac{b_i}{a_i},
            \end{align*}
    where $l_i=\log(b_ix_i+a_i)$, $a_i=e^{\sigma_i^2/2}$, $b_i=\sqrt{e^{2\sigma_i^2}-e^{\sigma_i^2}}$, $i=1,2$.
    
    \medskip
    
    \item a Sinh$^{-1}$-normal distribution Sinh$^{-1}$N($\mu_1,\mu_2,\sigma_1,\sigma_2,\rho$) with density
    \begin{align*}
        g_5(x_1,x_2;\mu_1,\mu_2,\sigma_1,\sigma_2,\rho)&=\frac{b_1b_2(w_1+\sqrt{1+w_1^2})(w_2+\sqrt{1+w_2^2})}{(1+w_1^2+w_1\sqrt{1+w_1^2})(1+w_2^2+w_2\sqrt{1+w_2^2})}\\&\times g_1(\sinh^{-1}(w_1),\sinh^{-1}(w_2);\mu_1,\mu_2,\sigma_1,\sigma_2,\rho),
            \end{align*}
    where $w_i=b_ix_i+a_i$, $a_i=e^{\sigma_i^2/2}\sinh(\mu_i)$, $b_i=\sqrt{(e^{\sigma_i^2}-1)(e^{\sigma_i^2}\cosh(2\mu_i)+1)}$, $i=1,2$.
    
    \medskip
    
    \item a generalized Burr-Pareto-Logistic distribution GBPL($\alpha,\beta$) with standard normal marginals, with density
    \begin{align*}
    g_6(x_1,x_2;\alpha,\beta)&=\frac{(\alpha+1)\varphi(x_1)\varphi(x_2)}{\alpha\Phi(x_1)\Phi(x_2)}\bigg(\frac{1+\beta}{(\Phi(x_1))^{-\frac1\alpha}+(\Phi(x_2))^{-\frac1\alpha}-1)^{\alpha+2}}\\&+\frac{4\beta}{2(\Phi(x_1))^{-\frac1\alpha}+2(\Phi(x_2))^{-\frac1\alpha}-3)^{\alpha+2}}\\&-\frac{2\beta}{2(\Phi(x_1))^{-\frac1\alpha}+(\Phi(x_2))^{-\frac1\alpha}-2)^{\alpha+2}}\\&-\frac{2\beta}{(\Phi(x_1))^{-\frac1\alpha}+2(\Phi(x_2))^{-\frac1\alpha}-2)^{\alpha+2}}\bigg),\; \alpha>0, \beta \in [-1,1],
    \end{align*}
    where $\Phi(x)$ and $\varphi(x)$ are the standard normal distributions function and density;
    
    \medskip
    
    \item a Morgenstern distribution Morg($\alpha$), , with standard normal marginals, with density
    \begin{align*}
        g_7(x_1,x_2;\alpha)=\varphi(x_1)\varphi(x_2)\Big(1+\alpha\big(2\Phi(x_1)-1\big)\big(2\Phi(x_1)-1\big)\Big), \;\alpha\in[-1,1];
    \end{align*}
        \item a Pearson type VII distribution PearVII($\alpha$) with density
        \begin{align*}
        g_8(x_1,x_2;\alpha)=\frac{\alpha}{2\pi}   \Big(1 + \frac12(x_1^2+x_2^2)\Big)^{\alpha + 1},\;\; \alpha>0. 
    \end{align*}
\end{itemize}

Methods of generating random variates from these distributions are available in \cite{johnson1987} and \cite{cook1986}.
From Table \ref{tab:bBivPow}  we can see that our new test is more powerful for the great majority of alternatives. In some cases, such as Normal Mixture B, and Normal Mixture A for negatively correlated components, the difference in powers is outstanding. On the other hand $KS2$ is consistently slightly better only for the Pearson VII alternative.
Worth mentioning is that the powers of our new test are symmetric with respect to the sign of correlation parameter $\rho$, which doesn't hold in general for the $KS2$ test.



\vspace{0.5cm}

\color{black}
\vspace{0.5cm}

\begin{table}[httb]
    \centering
       \caption{Powers for testing bivariate normality and independence}
    \begin{tabular}{|cr|cc|cc|cr|cc|cc|}
    \hline & & \multicolumn{2}{|c|}{$n=20$} &
    \multicolumn{2}{|c|}{$n=50$}& & & \multicolumn{2}{|c|}{$n=20$} &
    \multicolumn{2}{|c|}{$n=50$}\\
        \hline Alt. & $\rho$ & KS2 & $M_2$ & KS2 & $M_2$ & Alt. & $\rho$ & KS2 & $M_2$ & KS2 & $M_2$\\
        \hline
       \multirow{4}{*}{\shortstack{BivN\\($0,0,1,1,\rho$)}} & 0 & 5 & 5 & 5 & 5&\multirow{4}{*}{\shortstack{BivN\\($0,0,1,1,\rho$)}} &  &  &  &  & \\
       & 0.1 & 7 & 8 & 9 & 10&& -0.1 & 7 & 8 & 9 & 10\\
       & 0.3 & 24 & 25 & 46 & 50 && -0.3 & 24 & 25 & 46 & 50 \\
       & 0.5 & 60 & 64 & 94 & 95&& -0.5 & 60 & 64 & 94 & 95\\
       \hline
       \multirow{7}{*}{NMixA($\rho$)}& 0 & 64 & 68 & 93 &97 & \multirow{7}{*}{NMixB($\rho$)}& 0 & 6 &  8& 7&13 \\
       & 0.1 & 81 & 90 & 98 & 88& & 0.1 & 13 & 90 & 25 & 87\\
       & 0.3 & 90 & 94 & 100 & 97&& 0.3 & 14 & 94 & 26 & 97\\
       & 0.5 & 97 & 98 & 100 & 100&& 0.5 & 20 & 98 & 38 & 100\\
       & -0.1 & 66 & 90 & 95 & 91&& -0.1 & 12 & 90 & 23 & 88 \\
       & -0.3 & 56 & 93 & 88 & 96&& -0.3 & 14 & 93 & 23 & 96 \\
       & -0.5 & 51 & 98 & 80 & 100&& -0.5 & 21 & 97 & 36 & 100\\
       \hline
       
        \multirow{7}{*}{\shortstack{LogN\\($1,1,\rho$)}}& 0 &59  &73   &94 &100& \multirow{7}{*}{\shortstack{LogN\\($0.05,0.5,\rho$)}}& 0 & 58 &75 &  93&100\\
        & 0.1 & 60 & 84  & 93&100&&0.1 & 60 & 83 &94  &100\\
       & 0.3 & 65 & 87 &95 &100&& 0.3 & 66 & 88 &93 &100\\
       & 0.5 & 83 & 97 &96 &100&& 0.5 & 71 & 94 &96 &100\\
       & -0.1 & 62 & 85 &94 &100&& -0.1 & 59 & 83 &94 &100\\
       & -0.3 & 63 & 85 &96 &100&& -0.3 & 62 & 88 &96 &100\\
       & -0.5 & 70 & 93 &98 &100&& -0.5 & 70 & 93 &98 &100\\
       \hline
               \multirow{7}{*}{\shortstack{Sinh$^{-1}$N\\($0,0,1,1,\rho$)}}& 0 &32  &28   & 67 &60 & \multirow{7}{*}{\shortstack{Sinh$^{-1}$N\\($0,2,1,0.5,\rho$)}}& 0 &33  & 31  &70 &63\\
        & 0.1 & 33 & 35  & 69 & 74& & 0.1 & 32 & 36  &73 &76\\       
       & 0.3 & 43 & 50 & 78 & 80&& 0.3 & 41 & 48 &80 &92\\
       & 0.5 & 62 & 76 & 92&97&& 0.5 & 60 & 75 &95 &99\\
       & -0.1 & 32 & 37 &70 &67&& -0.1 & 33 & 37 & 69&74\\
       & -0.3 & 42 & 53 &79 &81&& -0.3 & 40 & 48 &78 &89\\
       & -0.5 & 62 & 74 &94&98&& -0.5 & 58 & 70 &92 &97\\
\hline\hline

        \hline Alt. & $\alpha$ & KS2 & $M_2$ & KS2 & $M_2$&Alt. & $\alpha$ & KS2 & $M_2$ & KS2 & $M_2$ \\
        \hline
        \multirow{4}{*}{GBPL($\alpha,-1$)}& 1 &  35 &  35 & 52 & 67&\multirow{4}{*}{GBPL($\alpha,1$)}& 1 & 87 & 88  & 100 & 100\\
       & 2 & 10 & 9 & 11 & 12&& 2 & 62 & 66 & 93 & 98\\
       & 5 & 8 & 9 & 9 & 12&& 5 & 43 & 48 & 67 & 82\\
       & 10 & 12 & 15 & 13 & 31&& 10 & 33 & 38 & 52 & 74\\
        \hline
            \multirow{6}{*}{Morg($\alpha$)}& 0.5 & 9 & 11 & 9 & 14& &&&&& \\
       & 0.75 & 15 & 18 & 18 & 33& \multirow{4}{*}{PearVII($\alpha$)}& 1 & 67 & 67 & 94 & 92\\
       & 1 & 22 & 28 & 34 & 59&& 2 & 34 & 30 & 65 & 54\\
       & -0.5 & 8 & 9 & 8 & 16&& 5 & 16 & 12 & 24 & 16\\
       & -0.75 &15 & 18 & 17 & 34&& 10 & 10 & 9 & 11 & 7\\
       & -1 & 24 & 29 & 35 & 57&&&&&&\\
       \hline
           \end{tabular}
 
    \label{tab:bBivPow}
\end{table}

\section*{Acknowledgments}

The work was supported by the Narodowe Centrum Nauki grant No 2018/29/B/HS4/01420 and the Ministry of Education, Science and Technological Development of the Republic of Serbia.


\end{document}